\documentclass{amsart}
\usepackage{amsmath,amssymb,graphicx}
\usepackage{mathrsfs}

\newtheorem{theorem}{Theorem}[section]
\newtheorem{lemma}[theorem]{Lemma}

\newtheorem{remark}[theorem]{Remark}
\newtheorem{definition}[theorem]{Definition}

\newcommand{\dotcup}{\dot{\cup}}
\newcommand{\cupn}[1]{\stackrel{#1}{\cup}}
\newcommand{\A}{\mathcal{A}}

\newcommand{\Lk}{\mathscr{L}_k}
\newcommand{\LO}{L^o}
\newcommand{\R}{\mathbb{R}}

\newcommand{\QQ}{\mathcal{Q}}
\newcommand{\T}{\mathcal{T}}

\begin{document}
	
\title{Ihara zeta functions for some simple graph families.}
\author[M.\ Chico]{Maize Chico}
\author[T.W.\ Mattman]{Thomas W.\ Mattman}
\email{TMattman@CSUChico.edu}
\author[D.A.\ Richards]{Alex Richards}
\email{(DAR) minerman60101@gmail.com}
\address{Department of Mathematics and Statistics,
California State University, Chico,
Chico, CA 95929-0525}

\begin{abstract}
   The reciprocal of the Ihara zeta function of a graph is a polynomial invariant introduced by Ihara in 1966.
   Scott and Storm gave a method to determine the coefficients of the polynomial. Here we simplify
   their calculation and determine the zeta function for all graphs of rank two. We verify that it is a complete
   invariant for such graphs: If $G_1$ and $G_2$ are of rank two, then $G_1$ and $G_2$ are isomorphic if and only if they have the
   same Ihara zeta function. We observe that the reciprocal of the zeta function is an even polynomial
   if the graph is bipartite.
   
    We also determine the zeta function for several graph families:  complete graphs, 
    complete bipartite graphs, M\"{o}bius ladders, cocktail party graphs, and all graphs of order five or less.
    We use the special value $u=1$ to count the spanning trees for these families.
\end{abstract}
	
\thanks{
The third author was supported in part by the CUReCAP program of the Office of Undergraduate Education
at CSU, Chico.
Maize Chico is Andrea Barnett, Alyxis Beardwood, James Cureton, Asher Curtis,
Natalie Dinin,
Emmanuel Disla,
Zaryab Fayyaz, Nicole Keane,
Ethan Marshburn, Drew Rodriguez,
and Logan Underwood.
}

\maketitle

\section{Introduction}

The Ihara zeta function was introduced by Yasutaka Ihara~\cite{I1,I2} in the 1960s. 
As discussed by Terras~\cite[Theorem 2.5]{T}, there is 
a determinant formula (due to Ihara, Bass, and Hashimoto among others) 
for the Ihara zeta function of a graph $G$, which is always the reciprocal of a polynomial.

\begin{definition} 
\label{def:3D}
Let $G$ be a connected graph with no vertices of degree 1.
Let $\A$ be the {\em adjacency matrix} of $G$, and $\QQ = D-I$ where $D$ is the {\em degree matrix}. 
That is,
$\QQ$ is diagonal with $j$th diagonal entry $q_j$ such that $q_j+1$ is the degree of the $j$th vertex of $G$. Let $r = |E|-|V|+1$ be the
{\em rank} of the graph.

Then,
 $$\zeta_G(u)^{-1} = \left(1-u^2\right)^{r-1} \det\left(I - \A u + \QQ u^2\right).$$
\end{definition}

 In \cite{SS}, Scott and Storm make use of an expression, due to Kotani and Sunada~\cite{KS}, for the
 zeta function using the oriented line graph, $\LO G$, of $G$. 
 
 \begin{definition}
 \label{def:T}
 Let $G$ be a connected graph with no vertices of degree 1
 and $\T$ be the adjacency matrix of $\LO G$. Then,
 $$\zeta_G(u)^{-1}=\det(I-u \T ).$$
 \end{definition}

We will sometimes refer to $\zeta_{G}(u)^{-1}$ as {\em the zeta polynomial} of graph $G$.
For us, {\em graphs} are undirected and, as in Definition~\ref{def:3D}, we are primarily interested
in connected graphs with no degree 1 vertices. Graphs may have loops or multiple edges
and, when they do, we call them {\em multigraphs}. In contrast, {\em simple graphs}
have no loops or double edges. We reserve the term {\em digraph} for directed graphs.

Our main theorem is a simplification of the method for calculating coefficients of the
Ihara zeta polynomial described by Scott and Storm~\cite{SS}. First we need two definitions.
\begin{definition}
    For a directed graph $D$, a {\em linear subgraph} is a disjoint collection of directed cycles.
\end{definition}

\begin{definition}
    For a directed graph $D$, let $\mathscr{L}_k(D)$ be the set of linear subgraphs of $D$ that have $k$ vertices.
For $L \in \mathscr{L}_k(D)$, $r(L)$ is the number of cycles in $L$.
\end{definition}

\begin{theorem}
\label{thm:Main}
    For a graph $G$, the reciprocal of the Ihara zeta function, $\zeta_G(u)^{-1}$, is a polynomial with terms $c_k u^k$ for 
    $0 \leq k \leq 2 |E|$.
    The constant term is $c_0 = 1$, and $c_k = 0$ when $\mathscr{L}_k(\LO G) = \emptyset$. Otherwise,
    \[
    c_k = \sum_{L \in \mathscr{L}_k(\LO G)} (-1)^{r(L)}.
    \]
\end{theorem}

A natural consequence of Scott and Storm's approach is a new proof that the Ihara zeta polynomial of a bipartite graph is even.
In fact, the converse also holds, see~\cite[Lemma 6]{Co}.

\begin{theorem}
\label{thm:even}
    If $G$ is bipartite, then the reciprocal of the Ihara zeta function, $\zeta_G(u)^{-1}$, is an even polynomial. That is $c_k = 0$ when
    $k$ is odd.
\end{theorem}

We prove these theorems in the next section.
Using Theorem~\ref{thm:Main}, we determine the zeta function for every graph of rank two in Section~\ref{sec:rk2}.
As in Definition~\ref{def:3D}, the rank of a graph is $|E|-|V| + 1$ where $|V|$ is the graph's {\em order}, or number of vertices,
and $|E|$ is its {\em size}, or number of edges.

In Section~\ref{sec:IZr2}, we show that the Ihara zeta function is a complete invariant for the rank two graphs:
\begin{theorem}
\label{thm:IZr2}
   Let $G_1$ and $G_2$ be graphs of rank two.
   Then, $G_1$ and $G_2$ are isomorphic if and only if $\zeta_{G_1}(u) = \zeta_{G_2}(u)$.
\end{theorem}

\begin{remark}
There are many examples of non-isomorphic graphs that have the same Ihara zeta function, the simplest
being a pair of graphs on eight vertices~\cite{SS2}, both of rank seven.
\end{remark}

Next, after introducing some lemmas for matrix determinants in Section~\ref{sec:Mdet}, in Section~\ref{sec:var}
we determine the zeta function for five families of symmetric or regular graphs: complete graphs, complete bipartite graphs,
cocktail party graphs, M\"obius ladders, and a family of bipartite graphs we call $B_n$. 
For $n$ even, $B_n$ is $K_{n/2,n/2}$ less the edges of a perfect matching.

In Section~\ref{sec:ord}, we give the zeta function for every multigraph of order three or less and every
simple graph with at most five vertices. 
In Section~\ref{sec:kappa}, we conclude the paper by 
using the special value $u=1$ of the Ihara zeta function to calculate the number of spanning trees for the graphs of
rank two and four of the five families of Section~\ref{sec:var}.

\section{Proof of Theorems~\ref{thm:Main} and \ref{thm:even}}

In this section we prove two theorems, beginning with
Theorem~\ref{thm:Main}, which is a simplification of the technique described by Scott and Storm~\cite{SS}.
We will assume some familiarity with that paper, but let's restate the main definitions. Recall that in a digraph, each edge $e$
has an {\em origin vertex} $o(e)$ and a {\em terminal vertex} $t(e)$. We often write $e = (o(e), t(e))$. Then the {\em inverse} of $e$ is
$\bar{e} = (t(e),o(e))$.

\begin{definition}
For a simple graph $G$, let $D(G)$ denote the associated {\em symmetric digraph}
on the same set of vertices and having $2|E(G)|$ directed edges. The {\em oriented line graph} of $G$ has 
vertices $\{E(D(G)\}$ (the edges of $D(G)$) and edges
$\{ (e_i,e_j) \in E(D(G)) \times E(D(G)) \mid \bar{e_i} \neq e_j, t(e_j) = o(e_i) \}$.
\end{definition}

\begin{proof} (of Theorem~\ref{thm:Main})
    In this argument, we largely use the ideas and notation of Storm and Scott's~\cite{SS} paper, with which we assume familiarity.
    Using Definition~\ref{def:T}, $\zeta_G(u)^{-1}=\det(I-u \T )$, where $\T$ is the adjacency matrix of the oriented line graph of G, $\LO G$.
    Storm and Scott~\cite{SS} use the characteristic polymonial of $\T$,
    $$\chi_{\T}(u) = \det( \T -uI) = u^{2m}+c_1 u^{2m-1} + ... + c_{2m},$$
    to find the following expression for the coefficients of the Ihara zeta function's reciprocal:
    $$\zeta_G(u)^{-1} = c_{2m} u^{2m}+c_{2m-1} u^{2m-1} + ... + c_1 u + 1.$$
    Note that $m = |E|$ is the number of edges of $G$.
    By~\cite[Lemma 12]{SS}, the coefficients of this characteristic polynomial, $c_i$, are given by $(-1)^i$ times the sum of all $i\times i$ principal minors of $\T$ (labelled $\det(\tilde{\T})$ in \cite{SS}). Next, \cite[Lemma 13]{SS} states that given a digraph $D$, with linear subgraphs $D_i$ for $i=1,...,n$, with $D_i$ having $e_i$ even cycles, then
    \begin{equation}
    \label{eqn:detA}
       \det(\A) = \sum_{i=1}^n (-1)^{e_i},
    \end{equation}   
    where $\A$ is the adjacency matrix of $D$. 
    Finally, in the proof of ~\cite[Theorem 7]{SS}, they consider $c_k$ for $2\leq k < 2m$. Then the $k \times k$ principal minors of $\T$ amount
    to choosing $k$ vertices of $\LO G$ and creating the subdigraph $\tilde{D}$ they induce. This subdigraph, $\tilde{D}$, 
    will be an element of $\mathscr{S}_k(\LO G)$, which is the set of subdigraphs on exactly $k$ vertices.

\begin{table}[b]
    \begin{center}
    
    \begin{tabular}{c|c|c|c}
         $k$ & $e(L)$ & $k + e(L)$ & $r(L)$ \\ \hline
         Even & Even & Even & Even \\
         Even & Odd & Odd & Odd \\
         Odd & Even & Odd & Odd \\
         Odd & Odd & Even & Even        
    \end{tabular} 
    
    \caption{
    \label{tbl:par}%
    Parities for $L$.}

    \end{center}
\end{table}
    
    Next, the principal minors will be the determinants of the adjacency minor $\tilde{T}$ of $\tilde{D}$. Let $\tilde{D}_i$ for $i=1,...,j$,
    be the linear subgraphs of $\tilde{D}$, with $\tilde{D}_i$ having $e(\tilde{D}_i)$ even cycles.
    Applying Equation~\ref{eqn:detA}, we have
    $$\det(\tilde{T}) = \sum_{\tilde{D}_i \subseteq \tilde{D}} (-1)^{e(\tilde{D}_i)}.$$
    Here is where we diverge from Scott and Storm to offer an alternate simplification of the zeta coefficients. 
    Summing over all principal minors:
    \begin{align*}
        c_k &=\sum_{\tilde{D} \in \mathscr{S}_k(\LO G)} (-1)^k \sum_{\tilde{D}_i \subseteq \tilde{D}} (-1)^{e(\tilde{D}_i)} \\
        &= \sum_{\tilde{D} \in \mathscr{S}_k(\LO G)}  \sum_{\tilde{D}_i \subseteq \tilde{D}} (-1)^k (-1)^{e(\tilde{D}_i)}
    \end{align*}
    As each $\tilde{D} \in \mathscr{S}_k(\LO G)$ will be formed from a unique set of $k$ vertices in $\LO G$, the set of all $\tilde{D}_i$ in the above double sum will be exactly $\mathscr{L}_k(\LO G)$ defined before, so we can simplify:
    \begin{align*}
        c_k &= \sum_{L \in \mathscr{L}_k(\LO G)} (-1)^k (-1)^{e(L)} \\
        &= \sum_{L \in \mathscr{L}_k(\LO G)} (-1)^{k + e(L)}
    \end{align*}
    Let $L \in \mathscr{L}_k(\LO G)$ and assume $k$ is odd. Then $k$ is the sum of the vertices in all cycles of $L$, so there must be an odd number of odd cycles. If the number of cycles $r(L)$ is even, then there are an odd number of even cycles, and if it is odd, then there are an even number of even cycles; the parity of $r(L)$ will be the opposite of $e(L)$. Similarly for $k$ even, there are an even number of odd cycles, so the parity of the number of even cycles $e(L)$ will match the parity of $r(L)$, 
    see Table~\ref{tbl:par}.
    
    As $r(L)$ and $k + e(L)$ have the same parity in all cases, we arrive at

    $$c_k = \sum_{L \in \mathscr{L}_k(\LO G)} (-1)^{r(L)} $$
\end{proof}

Next we prove Theorem~\ref{thm:even}, which is an observation that follows directly from Scott and Storm's~\cite{SS} approach.

\begin{proof} (of Theorem~\ref{thm:even})
The key observation is that a directed
cycle $\vec{C}$ in $\LO G$ corresponds to a closed path $C$ in $G$. 
This implies the $\vec{C}$ have even length when $G$ is bipartite as 
we now show.

Let $v_1, v_2, v_3, \ldots, v_n$ be the vertices of $\vec{C}$. Each $v_i$ corresponds to 
choosing an orientation of a corresponding edge $e_{v_i}$ in $G$. This means we
can make sense of the vertices $o(e_{v_i})$ and $t(e_{v_i})$ that are the induced
origin and terminus for $e_{v_i}$ under this orientation.
As we trace along $\vec{C}$ from $v_i$ to $v_{i+1}$, we move from $e_{v_i}$
to $e_{v_{i+1}}$ via the vertex $t(v_i) = o(v_{i+1})$. This means that the sequence of
edges $e_{v_1}, e_{v_2}, e_{v_3}, \ldots, e_{v_n}$ is a closed path $C$ that begins
and ends at $o(e_{v_1}) = t(e_{v_n})$.  

A closed path in a graph can be decomposed as the sum of a set of cycles and  
paths in the 
sense that they have the same edges. That is we can 
write $C = \gamma_1 + \gamma_2 + \cdots + \gamma_l + \rho_1 + \rho_2 + \cdots \rho_k$ 
where the $\gamma_i$ are cycles and the $\rho_j$ are paths. 
Note that each path is traversed twice (once in each direction) and,
since $G$ is bipartite, each cycle is even.
Consequently the length of $C$, $n$, is even, as is the length of $\vec{C}$.

We've argued that, when $G$ is bipartite, the directed cycles in $\LO G$ have even
length. This means that if $L \in \Lk (\LO G)$, (ie if $\Lk (\LO G)$ is nonempty)
then $k$ must be even as $L$ is
a disjoint union of directed cycles, each of which is of even length. 
By Theorem~\ref{thm:Main}, $c_k = 0$ when $k$ is odd.
\end{proof}

\section{Families of rank two graphs}
\label{sec:rk2}

In this section we use Theorem~\ref{thm:Main} to determine the zeta function for every graph of rank two. We begin by dividing the
simple graphs into three families, which, together, include every simple rank two graph.

\begin{definition}
Let $m,n > 2$.
A {\em double cycle graph} $G_{m,n} = C_m \dotcup C_n$, is the union of two cycles made by identifying one vertex of each.
The order and size are $|E(G_{m,n})| = m+n-1$ and $|V( G_{m,n} )| = m+n$.
\end{definition}

\begin{definition}
\label{def:Gmnp}%
Let $m,n > \max (p+1,2)$ with $p > 0$.
$G_{m,n,p} = C_m \cupn{p} C_n$, is the union of two cycles made by identifying $p$ consecutive edges of each.
 The order and size are $|E(G_{m,n,p})| = m+n-p-1$
and $|V(G_{m,n,p})| = m+n-p$.
\end{definition}


\begin{definition}
An {\em $H_{m,n,l}$ handcuff graph}, with $m,n>2$ and $l > 0$, is $C_m$ connected to $C_n$ by a path of $l$ edges. The order and size are 
$|V(H_{m,n,l})| = m+n+l-1$ and $|E( H_{m,n,l}) | = m+n+l$. 
\end{definition}


With the next three theorems, we give the Ihara zeta function for each of these families. In each case, we first state a lemma that
gives the structure of the oriented line graph $\LO G$.  

\begin{theorem}
\label{thm:Gmn}
For $3 \leq m \leq n$,
$$\zeta_{G_{m,n}}(u)^{-1} = -3u^{2(m+n)} + 2 u^{m+2n} + 2u^{2m+n} + u^{2n} + u^{2m} -2u^n -2u^m + 1.$$
\end{theorem}


To prove the theorem, we'll first determine the structure of the oriented line graph, 
$\LO G_{m,n}$.

\begin{lemma}
\label{lem:Gmn}
Let $3 \leq m \leq n$.
The oriented line graph $\LO G_{m,n}$ consists of a cube with four `wings'. 
Denote the vertices of the cube $v_{a,b,c}$ with $(a,b,c) \in \{0,1\}^3$
so that the $(a,b,c)$ are the corresponding points on a cube in $\R^3$.
Orient edges of the cube so that $v_{a,b,c}$ is a source (outdegree 3) if
$(a,b,c)$ has an even number of 1's and a sink (indegree 3) otherwise. 
The four wings are oriented cycles each using one of the vertical edges
of the cube. The wings alternate $C_m$ and $C_n$ cycles as we go
through the four vertical edges of the cube.
\end{lemma}

As an example, Figure~\ref{fig:LOG34} (produced using Sage~\cite{S}) shows $\LO G_{3,4}$.
\begin{figure}[htb]
\centering
\includegraphics[scale = 0.5]{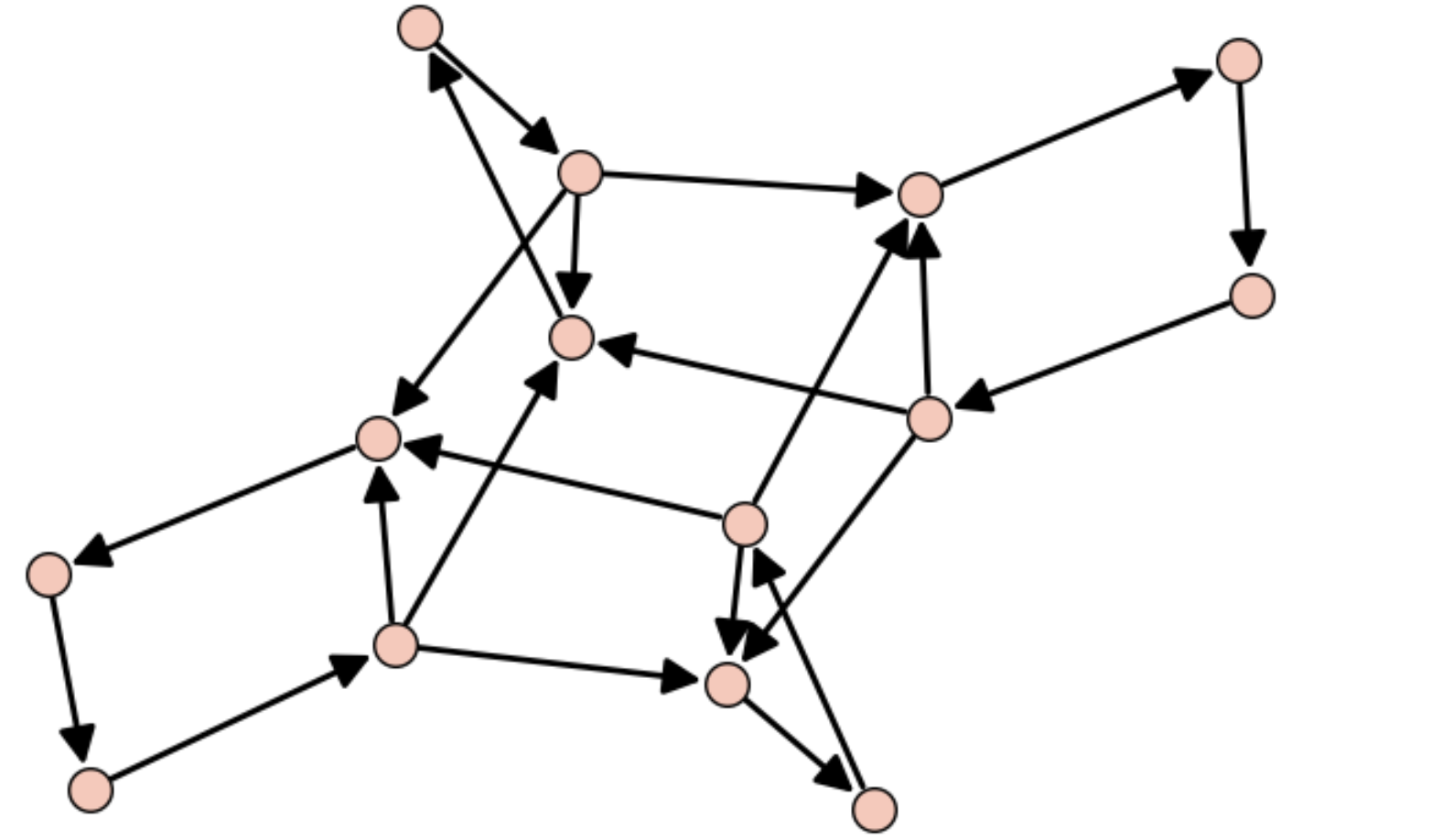}
\caption{The oriented line graph $\LO G_{3,4}$}
\label{fig:LOG34}
\end{figure}

\begin{proof}
In $G_{m,n}$, let $x$ denote the common vertex of the two cycles and 
$v_2, \ldots, v_{m}$
and $w_2, \ldots, w_n$ the remaining vertices of each cycle, in order as
we traverse the cycles. Then $N(x) = \{v_2, v_m, w_2, w_m\}$ and the 
vertices of the cube in $\LO G_{m,n}$ are the eight directed edges 
incident on $x$ in the symmetric digraph $D(G_{m,n})$. The four edges
that terminate at $x$ are the four sources in the cube and those that
originate at $x$ are sinks.

The remaining edges of $C_m$ and $C_n$ generate the wings. Each of the 
two directed cycles around $C_m$ in $D(G_{m,n})$ generates one wing in $\LO G_{m,n}$
and similarly for $C_n$.
\end{proof}

\begin{proof} (of Theorem~\ref{thm:Gmn}).
To prove the theorem we need to study the linear subgraphs of $\LO G_{m,n}$. Since
each vertex of the cube is either a source or a sink, in a linear subgraph
each cube vertex has at most one edge from the cube. Then it must have exactly one edge from the cube and one edge from a wing.
It follows that a cycle in a linear subgraph is either a wing, a $C_{m+n}$ cycle
using top and bottom edges of a cube face, or else a $C_{2(m+n)}$ cycle that uses
all four wings.

By Theorem~\ref{thm:Main}, for each $k$, we must analyze the linear subgraphs
on $k$ vertices, $\Lk (\LO G)$.
Given our listing of the cycles of $\LO G$ above, the possible values of $k$ are
in $\{ 2(m+n), m+2n, 2m+n, 2m, 2n, m, n\}$. Below we consider each of
these values in turn and determine $c_k$.

For $k = 2(m+n)$, there are four ways to make a linear subgraph: 
A single $2(m+n)$ cycle, two $m+n$ cycles, an $m+n$ cycle with two wings 
(one a $C_m$, the other a $C_n$), or else by using all four wings.
	
	\begin{table}[htb]
	    \centering
	    \begin{tabular}{l|c|c|c}
            Type of $L$ & Occurrences & $r(L)$ & Contribution\\ \hline
            $2(m+n)$ cycle  & 2 & 1 & -2 \\ \hline
            Two $m+n$ cycles & 2 & 2 & 2 \\ \hline
            $m+n$ cycle and & 4 & 3 & -4 \\
            two wings & & & \\ \hline
            four wings & 1 & 4 & 1
	    \end{tabular}
	    \caption{Linear subgraphs for $k = 2(m+n)$}
	    \label{tab:2m2nCase}
	\end{table}

Table~\ref{tab:2m2nCase} summarizes our calculation of $c_{2(m+n)} = -2 + 2 -4 + 1 = -3$.
For each linear subgraph $L$, we give the number of occurrences of $L$ in $\LO G_{m,n}$
(that is the number of subgraphs of $\LO G_{m,n}$ isomorphic to $L$)
and the rank, $r(L)$, or number of cycles in $L$. By Theorem~\ref{thm:Main}, each occurence
of $L$ contributes $(-1)^{r(L)}$ to $c_{2(m+n)}$.

	\begin{table}[htb]
	    \centering
	    \begin{tabular}{l|c|c|c}
            Type of $L$ & Occurences & $r(L)$ & Contribution \\ \hline
            $m+n$ cycle & 4 & 2 & 4 \\
            and $n$ wing & & & \\ \hline
            three wings & 2 & 3 & -2
	    \end{tabular}
	    \caption{Linear subgraphs for $k = m+2n$, $m$ even}
	    \label{tab:m2nCase}
	\end{table}

For $k = m+2n$, Table~\ref{tab:m2nCase} illustrates the calculation $c_{m+2n} = 4 -2 = 2$.
Determining that $c_{n+2m} = 2$ as well is analogous.

A linear subgraph of order $2m$ necessarily consists of the two $C_m$ wings and
such a subgraph occurs only once, so $c_{2m} = (-1)^2 = 1$.
Similarly, $c_{2n} = 1$.
A linear subgraph of order $m$ is a single wing, but there are two occurences
of such a wing in $\LO G_{m,n}$: $c_m = (-1)^1 + (-1)^1  = -2$ and,
similarly, $c_n = -2$.
Finally, recall (for example, see \cite{SS,T}) that the constant coefficient for $\zeta_G(u)^{-1}$ is 1
for any graph $G$.
\end{proof}

\begin{theorem}
\label{thm:Gmnp}
For $m,n > \max (p+1,2)$ with $p > 0$,
\begin{align*}
    \zeta_{G_{m,n,p}}(u)^{-1} =& -4u^{2m+2n-2p} + u^{2m+2n-4p}+2u^{m+2n-2p}
+2u^{2m+n-2p} \\
& \mbox{ } + u^{2n} + u^{2m} + 2u^{m+n} -2u^{m+n-2p} -2u^n-2u^m+1.
\end{align*}
\end{theorem}

\begin{remark}
This reproduces \cite[Corollary 19]{SS} when $p=1$, after replacing $m$ by $k+1$ and
$n$ by $n-k+1$.
\end{remark}


\begin{lemma}
The oriented line graph $\LO G_{m,n,p}$ consists of two hexagons,
which we'll call the `top' and `bottom' hexagon, along with six directed vertical paths.
Proceeding around the top hexagon, one alternates
sink and source vertices. Each sink vertex is also the origin of a `vertical' directed path to the bottom
hexagon and each source vertex is also the terminus of a vertical directed path from the bottom
hexagon. The edge lengths of these vertical paths cycle
through $p-1, m-p-1, n-p-1, p-1, m-p-1, n-p-1$. Similarly, moving around
the bottom hexagon, we alternate between sink and source vertices.
\end{lemma}

\begin{figure}[htb]
\centering
\includegraphics[scale = 0.5]{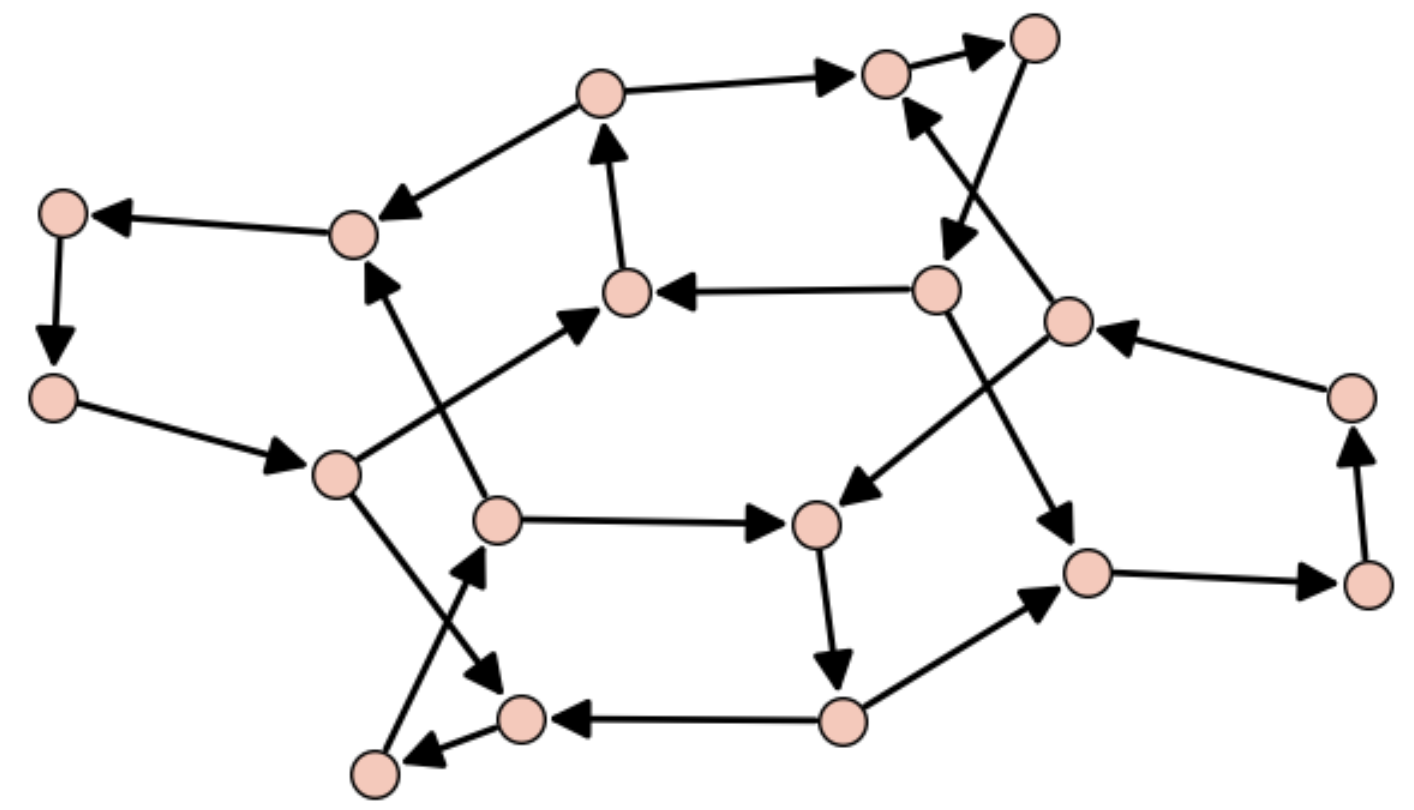}
\caption{The oriented line graph $\LO G_{5,6,2}$}
\label{fig:LOG563}
\end{figure}

\begin{figure}[htb]
\centering
\includegraphics[scale = 0.5]{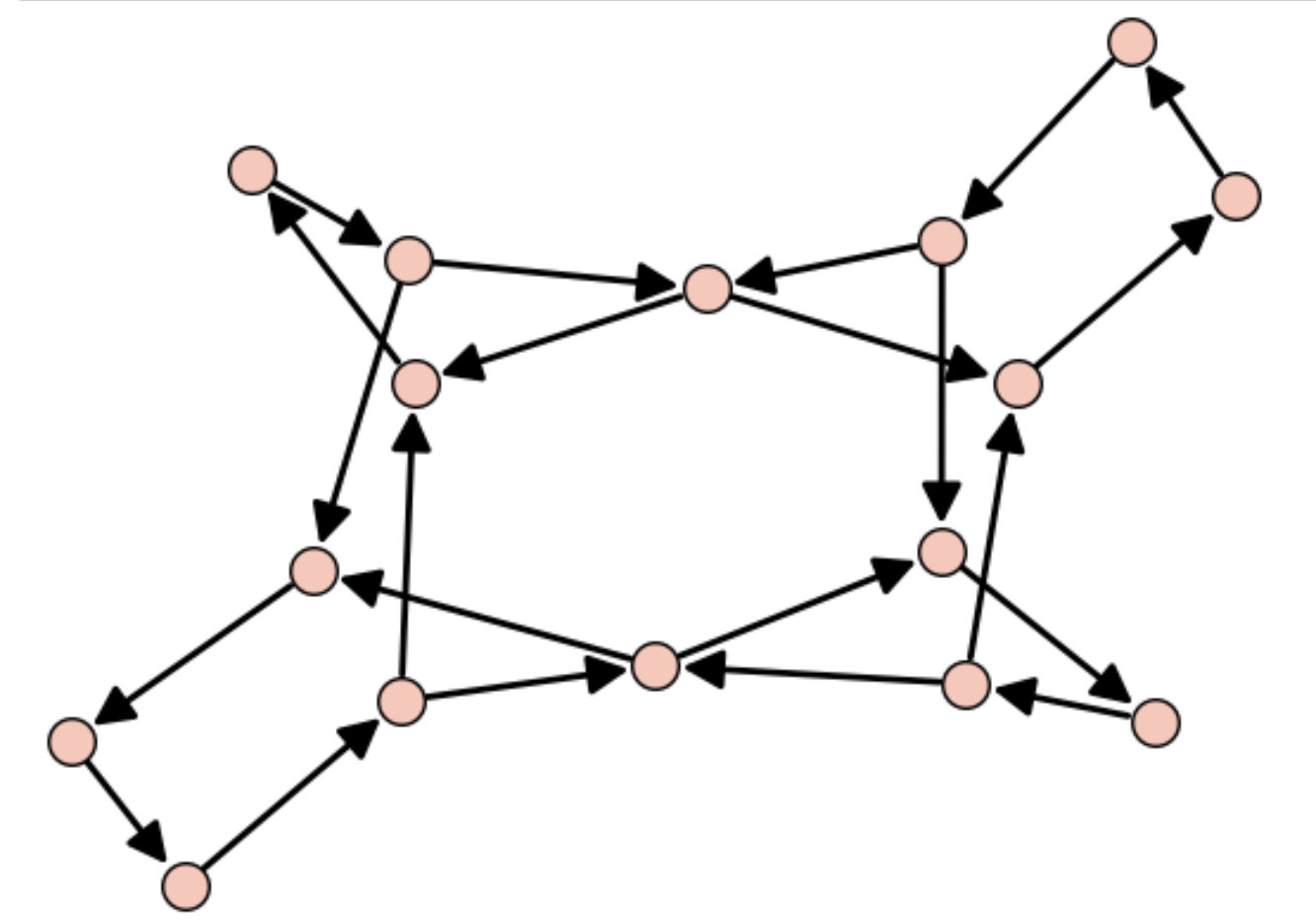}
\caption{The oriented line graph $\LO G_{4,5,1}$}
\label{fig:LOG452}
\end{figure}

Figures~\ref{fig:LOG563} and \ref{fig:LOG452} (produced using Sage~\cite{S}) are
representative. In Figure~\ref{fig:LOG452}, which shows $\LO G_{4,5,2}$, the
hexagons are vertical and might better be called `front' and `back.' Since $p = 1$,
the two hexagons are pinched together with a vertex identified at top and bottom.
We leave a proof of the lemma to the reader. The argument is similar to the proof of Lemma~\ref{lem:Gmn}.

\begin{proof} (of Theorem~\ref{thm:Gmnp})
Similar to the proof of Theorem~\ref{thm:Gmn}, since the hexagon vertices are sinks or sources, each linear
graph through the vertex uses one edge of the hexagon and one edge of its vertical path.
It follows that a directed cycle in a linear subgraph is either a `wing' that runs up and down 
between the two hexagons using neighboring vertical paths attached with
one edge in the top hexagon and one in the bottom, or else it is a 
$2(m+n-p)$-cycle that passes through all the vertical paths, alternating one edge in the top
hexagon and one edge in the bottom hexagon to move from one vertical path to the next.
The lengths of the wings are $m$  (using a $p-1$ and an $m-p-1$ vertical path), 
$n$ (including $p-1$ and $n-p-1$ paths), and $m+n-2$ ($m-p-1$
and $n-p-1$ paths).

Given our listing of the cycles of $L$ above, the possible values of $k$ are
in 
$$\{ 2(m+n-p),  2(m+n-2p), m+2n-2p, 2m+n-2p, 2n, 2m, m+n, m+n-2p, m+n, n, m \}.$$ 
Below we consider each of these values in turn and determine $c_k$.

	\begin{table}[htb]
	    \centering
	    \begin{tabular}{l|c|c|c}
            Type of $L$ & Occurrences & $r(L)$ & Contribution\\ \hline
            $2(m+n-p)$ cycle  & 2 & 1 & -2 \\ \hline
            $m+n-2p$ cycle and & 2 & 3 & -2 \\
            $m$ and $n$ wing & & & 
	    \end{tabular}
	    \caption{Linear subgraphs for $k = 2(m+n-p)$}
	    \label{tab:2mnpCase}
	\end{table}
Table~\ref{tab:2mnpCase} shows that $c_{2(m+n-p)} = -4$.

For $k = 2(m+n-2p)$, there is only one $L$ that uses the two $m+n-2p$ wings, so $c_2(m+n-2p) = (-1)^2 = 1$.
Similarly, when $k = 2n$ or $2m$, we use the two $n$ wings or $m$ wings and $c_k = 1$.

For $k = m + 2n - 2p$ (and similarly, for $2m+n-p$), $L$ is the combination of an
$m+n-p$ wing and an $n$ wing. This subgraph $L$ 
occurs twice in $\LO G_{m,n,p}$, so $c_{m+2n-2p} = (-1)^2 + (-1)^2  = 2$. 
Similarly, for $k = m+n$ we combine an $m$-wing and an $n$-wing, this arrangement
occurring twice, so $c_{m+n} = 2$.

Finally for $k \in \{m+n-2p, n, m\}$, $L$ is a single wing that occurs twice
in $\LO G_{m,n,p}$, so that $c_k = (-1)^1 + (-1)^1 = -2$.
\end{proof}



\begin{theorem}
\label{thm:Hmnl}
For $m,n>2$ and $l > 0$,
\begin{align*}
 \zeta_{H_{m,n,l}}(u)^{-1} =& -4u^{2m+2n+2l}+u^{2m+2n} + 4u^{2m+n+2l} + 4u^{m+2n+2l}-2u^{2m+n}-2u^{m+2n} \\
 & \hspace{0.5 in} -4u^{m+n+2l}+ 4u^{m+n} +u^{2n} + u^{2m} - 2u^n-2u^m+1.
 \end{align*}
 \end{theorem}

\begin{lemma}
The oriented line graph $\LO H_{m,n,l}$ consists of two hexagons,
which we'll call the `top' and `bottom' hexagons, two directed vertical paths, 
and four wings. 
Proceeding around the top hexagon, one alternates
sink and source vertices. At antipodal vertices $v_1$ and $v_4$ there is a vertical path that, at the sink
vertex, goes down to the bottom hexagon and at the source vertex is the terminus
of a vertical path up from the bottom hexagon. The vertical paths have $l-1$ edges.
Each of the two edges of the top hexagon that are adjacent to neither $v_1$ nor $v_4$ 
is part of an $n$-wing. The situation on the bottom hexagon is similar except that it has
two $m$-wings.
\end{lemma}

\begin{figure}[htb]
\centering
\includegraphics[scale = 0.5]{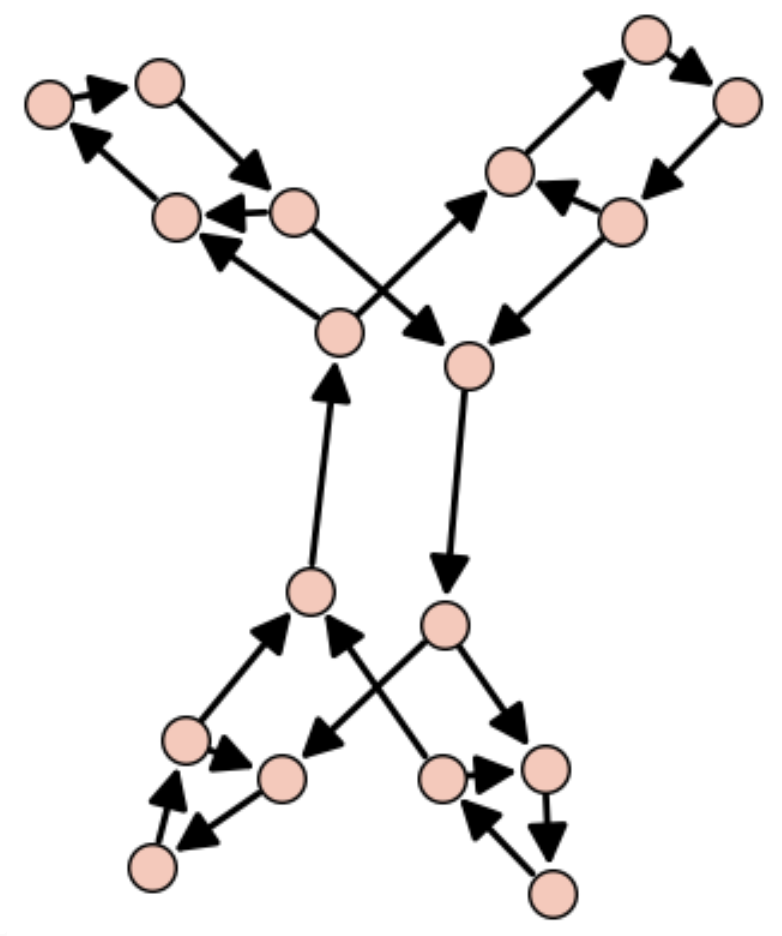}
\caption{The oriented line graph $\LO H_{4,3,2}$}
\label{fig:LOH432}
\end{figure}

For example, Figure~\ref{fig:LOH432} (constructed using Sage~\cite{S}) shows $\LO H_{4,3,2}$.
In the figure, the two hexagons have been folded into `V' shapes.
We leave the proof of the lemma to the reader.

\begin{proof} (of Theorem~\ref{thm:Hmnl})
As in the earlier proofs, cycles through a hexagon vertex must use exactly one edge of the hexagon 
with the other edge being part of the wing or the vertical path at that vertex.
This means there are three types of directed cycles, $m$ wings, $n$ wings, and a
$m+n+2l$-cycle that uses both of the vertical paths. The possible values
for $k$ are in $\{2(m+n+l), 2(m+n), m+2n + 2l, 2m+n + 2l, m+2n, 2m+n, m+n+2l, m+n, 2n, 2m, n,m\}$.

For $k = 2(m+n+l)$, the linear graph is an $m+n+2l$-cycle with an $m$ wing and an $n$ wing. This graph
occurs four times in $\LO H_{m,n,l}$, so $c_{2(m+n+l)} = (-1)^3 + (-1)^3 + (-1)^3 + (-1)^3 = -4$.

For $k = 2(m+n)$, we use both $m$-wings and both $n$-wings. There's only one occurrence, 
so $c_2(m+n) = (-1)^4 = 1$.

For $k = m+ 2n+ 2l$ (and similarly for $2m+n+2l$) $L$ consists of an $m+n-2l$ cycle and an $n$ wing.
There are four occurrences of $L$ so $c_{m+2n+2l} = (-1)^2 + (-1)^2 + (-1)^2 + (-1)^2$.

For $k = m+2n$ (similarly for $2m + n$) we use both $n$ wings and one of the $m$ wings. There are
two occurrences so $c_{m+2n} = (-1)^3 + (-1)^3 = -2$. 

For $k = m+n+2l$ we use just the one cycle, but there are four occurrences, so 
$c_{m+n+2l} = (-1)^1 + (-1)^1 + (-1)^1 + (-1)^1 = -4$

For $k = m+n$ we use one $m$-wing and one $n$-wing and there are four
occurrences, so $c_{m+n} = (-1)^2 + (-1)^2 + (-1)^2 + (-1)^2 = 4$.

For $k = 2n$ (or $2m$) we use both $n$ wings and there's only one occurrence, 
so $c_{2n} = (-1)^2 = 1$.
Finally, for $k = n$ (or $m$) we use a single $n$ wing and there are two occurrences,
so $c_n = (-1)^1 + (-1)^1 = -2$.
\end{proof}


We next run through the possible multigraphs of rank two. It turns out these can all be identified
as specializations of the formulas in our three Theorems.

For $n \geq 3$, let $G_{1,n}$
denote an $n$-cycle together with a loop at one vertex and $G_{2,n}$ be an $n$-cycle
with a bigon attached at one vertex. Here, a {\em bigon} is a 2-cycle or, in other words, 
a graph with 2 vertices and a double edge between them.

\begin{theorem}
For $3 \leq n$,
$$\zeta_{G_{1,n}}(u)^{-1} = -3u^{2(1+n)} + 2 u^{1+2n} + 2u^{2+n} + u^{2n} + u^{2} -2u^n -2u + 1,$$
and
$$\zeta_{G_{2,n}}(u)^{-1} = -3u^{2(2+n)} + 2 u^{2+2n} + 2u^{4+n} + u^{2n} + u^{4} -2u^n -2u^2 + 1.$$
\end{theorem}

We omit the proof which is similar to the proof of Theorem~\ref{thm:Gmn}.

\begin{remark} These formulas agree with setting $m=1,2$ in Theorem~\ref{thm:Gmn}. \end{remark}

Similarly, the rank two graphs that have only loops and multiedges agree with the 
expected specialization of the formula for $G_{m,n}$. Specifically, let $BQ_2$ denote
a bouquet of two loops, $BL$ be a bigon with a loop and $BB$ be two bigons joined
at a single vertex. Then 
\begin{align*}
\zeta_{BQ_2}(u)^{-1} &= \zeta_{G_{1,1}}(u)^{-1} = 
-3u^4 + 4u^3 + 2u^2 - 4u + 1 \\
\zeta_{BL}(u)^{-1} &= \zeta_{G_{1,2}}(u)^{-1} = 
-3u^6 + 2u^5 + 3u^4 - u^2 - 2u + 1 \\
\zeta_{BB}(u)^{-1} &= \zeta_{G_{2,2}}(u)^{-1} = 
-3u^8 + 4u^6 + 2u^4 - 4u^2 + 1
\end{align*}

\begin{remark} These three graphs have three or fewer vertices and we can also calculate the
Ihara zeta function using the methods of Section~\ref{sec:ord}. 
\end{remark}

Let $G(2,n,1)$ denote a graph that is a $C_m$ with one doubled edge. The zeta function is a specialization
of that for $G(m,n,p)$:

\begin{align*}
    \zeta_{G_{2,n,1}}(u)^{-1} =& -4u^{2n+2} + 4u^{2n}+4u^{n+2} + u^{4} -4u^n-2u^2+1.
\end{align*}

In particular, for the theta graph $G_{2,2,1}$ we have

\begin{align*}
    \zeta_{G_{2,2,1}}(u)^{-1} =& -4u^{6} + 9u^{4} -6u^2+1.
\end{align*}
Again, we can also calculate the zeta function for the theta graph using the techniques of 
Section~\ref{sec:ord}.

Let $H_{2,n,l}$ denote a Handcuff graph where the cycle $C_n$ is joined
to a bigon by a path of $l$ edges. The polynomial for $H_{2,n,l}$ is a special 
case of that for $H_{m,n,l}$ by substituting $m = 2$.

\begin{align*}
 \zeta_{H_{2,n,l}}(u)^{-1} =& -4u^{2(2+n+l)}+u^{2(2+n)}+4u^{2+2n+2l} + 4u^{4+n+2l}-2u^{4+n}-2u^{2+2n} \\
 & \hspace{0.5 in} -4u^{2+n+2l}+u^{2n} + 4u^{2+n} + u^{4} - 2u^n-2u^2+1
 \end{align*}
 
Let $H_{1,n,l}$ denote a Handcuff graph where the cycle $C_n$ is connected to a path of $l$ edges
with a loop at its other end. The polynomial for $H_{1,n,l}$ is a special 
case of that for $H_{m,n,l}$, by substituting $m = 1$.

\begin{align*}
 \zeta_{H_{1,n,l}}(u)^{-1} =& -4u^{2(1+n+l)}+u^{2(1+n)}+4u^{1+2n+2l} + 4u^{2+n+2l}-2u^{2+n}-2u^{1+2n} \\
 & \hspace{0.5 in} -4u^{1+n+2l}+u^{2n} + 4u^{1+n} + u^{2} - 2u^n-2u+1
 \end{align*}
 
 There remain three related types of graph where we have a path of length $l$ with a loop or a bigon 
 at both ends. Let $H_{2,2,l}$ denote two bigons joined by a path of length $l$. Again, the polynomial
 specializes that for $H_{m,n,l}$.

\begin{align*}
 \zeta_{H_{2,2,l}}(u)^{-1} =& -4u^{2(4+l)}+u^{8}+ 8u^{6+2l}-4u^{4+2l} - 4u^{6} + 6u^{4} - 4u^2 +1
\end{align*}
 
 If we have a bigon and a loop joined by a path of length $l$, we'll write $H_{1,2,l}$. As usual, 
 the polynomial specializes $H_{m,n,l}$.
 
\begin{align*}
 \zeta_{H_{1,2,l}}(u)^{-1} =& -4u^{2(3+l)}+u^{6}+4u^{4+2l} + 4u^{5+2l}-2u^{5}-u^{4} -4u^{3+2l} + 4u^{3} - u^2-2u+1
 \end{align*}
 
 Finally, if $H_{1,1,l}$ denotes two loops joined by a path of length $l$, we again
 specialize $H_{m,n,l}$.
 
 \begin{align*}
 \zeta_{H_{1,1,l}}(u)^{-1} =& -4u^{2(2+l)}+u^{4}+8u^{3+2l}-4u^{3} -4u^{2+2l} + 6u^{2} -4u+1
 \end{align*}
 
Notice that the bipartite rank two graphs have even zeta polynomial, in agreement with Theorem~\ref{thm:even}. 
For example, a double cycle graph $G_{m,n}$ is bipartite exactly if $m$ and $n$ are both even, and, in that case, 
the polynomial of Theorem~\ref{thm:Gmn} is even. Similarly, a $G_{m,n,p}$ or $H_{m,n,l}$ graph
is bipartite exactly when $m$ and $n$ are both even and, in that case, the zeta polynomial is even.

\section{Rank two graphs with equal zeta function}
\label{sec:IZr2}

In this section we prove Theorem~\ref{thm:IZr2}: if $\Gamma_1$ and $\Gamma_2$ are
rank two graphs that share the same Ihara zeta function, 
then $\Gamma_1$ and $\Gamma_2$ are isomorphic.

\begin{proof}
We begin by summarizing some aspects of the Ihara zeta function that can be determined
directly from the graph.

Kotani and Sunada~\cite{KS} showed that the leading coefficient of $\zeta_{G}(u)^{-1}$ is
$$c_{2|E|} = (-1)^{|E|-|V|} \prod_{v_i \in V}(d(v_i) - 1).$$ We have seen that aside from vertices of 
degree 2 (which do not contribute to the product), a rank two graph is either a $G_{m,n}$ graph
with a degree 4 vertex and $c_{2|E|} = -3$ or else a
$H_{m,n,l}$ handcuff graph or a $G_{m,n,p}$ graph, with two degree 3 vertices and $c_{2|E|} = -4$.
Hence if two rank two graphs produce the same zeta function, they must have the same size
$|E|$ and both have the same leading coefficient $-3$ or $-4$.

As discussed by Scott and Storm~\cite{SS} (see Theorem 5), for a simple graph, the first non-zero 
coefficient after $c_0=1$, is $c_g$, where $g$ is the graph's girth.  
Recall that the {\em girth} of a graph is the size of its shortest cycle.
We can generalize the definition of girth to multigraphs by saying
that a graph with a loop has girth one. If a multigraph has no loops,
but does have bigons, we'll say the girth is two. Corollary 14 of
\cite{SS} shows that two multigraphs with the same Ihara zeta function
must have the same girth. This implies that a simple graph cannot share
its Ihara zeta function with a graph that has loops or bigons.

A final way that we can go between a graph and its zeta function is through
the count of the number of spanning trees, denoted $\kappa_G$.
As in Theorem~\ref{thm:ASTF} below, for a rank two graph,
$$ \kappa_G = - \frac{1}{8} \frac{d^2}{du^2} \zeta_G(u)^{-1}\bigg\rvert_{u = 1}.$$
If two rank two graphs have the same zeta function, they must also have the
same number of spanning trees.

We begin with the case where two graphs have the same leading 
coefficient, $-3$.
Suppose $\Gamma_1$ and $\Gamma_2$ of rank two share the 
same Ihara zeta function, both having leading coefficient $-3$
for $\zeta_{\Gamma_i}(u)^{-1}$.
If both are simple, then they are both $G_{m,n}$ graphs. 
As $\Gamma_1$ and $\Gamma_2$ have
the same order and girth, they must be isomorphic. 

The non-simple 
rank two graphs with leading coefficient $-3$ are graphs of 
the form $G_{2,n}$ or $G_{1,n}$. 
Suppose $\Gamma_1$ is not simple and has an Ihara zeta function so that
$\zeta_{\Gamma_1}(u)^{-1}$ has leading coefficient $-3$. 
If $\Gamma_2$ has the same zeta function, then, since they must have the same girth, $\Gamma_2$ is also not simple.

If $\Gamma_1$ has girth one, it has a loop and $\Gamma_2$ must have a loop too. Since
they have the same size, $\Gamma_1$ and $\Gamma_2$ are isomorphic. Similarly, if
$\Gamma_1$ has a bigon and no loop, then $\Gamma_2$ must also have a bigon and no loop. 
Having the same size, $\Gamma_1$ and $\Gamma_2$ are again isomorphic in this case. 
In summary, if $\Gamma_1$ and $\Gamma_2$ both of rank two and have the same Ihara zeta function
with leading coefficient $-3$, then $\Gamma_1$ and $\Gamma_2$ are isomorphic.

Next suppose $\Gamma_1$ is a simple graph of rank two so that the leading 
coefficient of $\zeta_{\Gamma_1}(u)^{-1}$ is $-4$.
Then $G$ is a $G_{m,n,p}$ (with $p > 0$)
or a $H_{m,n,l}$ (with $l > 0$). Suppose $\Gamma_1$ is a $G_{m,n,p}$. 

Further, we may assume $0 < 2p \leq m \leq n$ so that $\Gamma_1$ has girth $m$.
Indeed, we can think of $G_{m,n,p}$ as a theta graph with distinct vertices $v$ and $w$ 
joined by three paths, of $m-p$, $n-p$, and $p$ edges, respectively. Without loss of 
generality, we can order these paths so that $n-p \geq m-p \geq p$, which gives the
required relationship $2p \leq m \leq n$.

If $\Gamma_2 = H_{m',n',l}$ with $m' \leq n'$ and $l > 0$
has the same Ihara zeta function as $\Gamma_1$,
then $\Gamma_2$ has girth $m'$, so $m' = m$. Since they
have the same size, we conclude $n' + l = n-p$ or
$n = n' + l + p$. Counting spanning trees (see Section~\ref{sec:kappa}) we
have $mn' = mn-p^2$, whence 
$p^2 = m(l+p)$. Since $l> 0$ and $m \geq 2p$ this is a contradiction.

Suppose instead that $\Gamma_2 = G_{m',n',p'}$, with 
$0 < 2p' \leq m' \leq n'$, has the same Ihara zeta function as $\Gamma_1$.
Since $\Gamma_1$ and $\Gamma_2$ have the same girth, $m' = m$. 
They must also have the same number of edges, so $m+n-p = m'+n'-p'$, and $n-p = n'-p'$. 

Let's examine the zeta function of $\Gamma_1$.
\begin{align*}
    \zeta_{\Gamma_1}(u)^{-1} =& -4u^{2m+2n-2p} + u^{2m+2n-4p}+2u^{m+2n-2p}
+2u^{2m+n-2p} \\
& \mbox{ } + u^{2n} + u^{2m} + 2u^{m+n} -2u^{m+n-2p} -2u^n-2u^m+1.
\end{align*}
Notice that we can pick out, not only the girth, which corresponds to the 
term $-2u^m$, but also the two next smallest cycle lengths, $n$ and $m+n-2p$,
identified by the terms $-2u^n$ and $-2u^{m+n-2p}$.
However, these terms can combine. For example, if $n = m$, then instead of $-2u^n$ and
$-2u^m$ we'll have the single term $-4u^n$. Since we're assuming $2p \leq m \leq n$, 
it follows that $m+n-2p \geq n \geq m$.

Suppose that $n = m$. Then the term corresponding to the girth $m$ is either $-4u^m$ or 
$-6u^m$ (in case $m+n-2p = m$ as well). In either case, $\zeta_{\Gamma_2}^{-1}$ has
the same penultimate term, so we conclude $n' = n$, whence $p' = p$. In this case
$\Gamma_1$ and $\Gamma_2$ are isomorphic graphs. 

On the other hand, if $n \neq m$, that is, if $n > m$, then $m+n-2p \geq n > m$ as well
and the penultimate term of $\zeta_{\Gamma_1}(u)^{-1}$ is just $-2u^m$.  
The term before that is either $-2u^n$ or else $-4u^n$ in case $m+n-2p = n$. In
either case, $\zeta_{\Gamma_2}^{-1}$ ends with the same terms, 
so we conclude $n' = n$, whence $p' = p$. In this case
$\Gamma_1$ and $\Gamma_2$ are isomorphic graphs. 

Next, suppose both $\Gamma_1$ and $\Gamma_2$ are handcuff graphs. Say $\Gamma_1 = H_{m,n,l}$
with $m \leq n$ and $l > 0$ and $\Gamma_2 = H_{m',n',l'}$ with $m' \leq n'$ and $l>0$. Since
they have the same girth, $m = m'$. Since they have the same number of spanning trees 
$mn = m'n$ (see Section~\ref{sec:kappa}), whence $n' = n$. Finally, we can deduce $l' = l$ as
the two graphs have the same number of edges. Thus, $\Gamma_1$ and $\Gamma_2$ are
isomorphic in this case as well.

Next, suppose $\Gamma_1$ is a multigraph of rank two whose zeta polynomial begins
with $-4u^{2|E|}$. As mentioned earlier in the proof, we can generalize the notion of girth
to multigraphs. Suppose $\Gamma_1$ has a loop and girth one. If $\Gamma_2$ of 
rank two has the same zeta function, it too has girth one and a loop. 
The only multigraphs of rank two with two degree 3 vertices and a loop are $H_{1,n,l}$ 
graphs. Suppose $\Gamma_1 = H_{1,n,l}$ with $n \geq 1$ and $l > 0$ and 
$\Gamma_2 = H_{1,,n'.l'}$ with $n' \geq 1$ and $l'>0$. Since they must have the 
same number of spanning trees (see Section~\ref{sec:kappa}), $n'=n$. As they 
have the same number of edges, $1+n+l = 1+n+l'$, so $l' = l$ as well and 
$\Gamma_1$ and $\Gamma_2$ are isomorphic.

Finally suppose $\Gamma_1$ is a multigraph of girth two, that is, it has a bigon
but no loop. If $\Gamma_2$ has the same zeta function, it too has a bigon and no loop.
So, $\Gamma_1$ and $\Gamma_2$ are either $G_{2,n,1}$ or $H_{2,n,l}$. 
Since $\zeta_{G_{2,n,1}}(u)^{-1}$ terminates with $u^4 -4u^n - 2u^2 + 1$ and
$\zeta_{H_{2,n,l}}(u)^{-1}$ with $u^4 -2 u^n - 2u^2+1$, there's no way that 
a $G_{2,n,1}$ could have the same zeta function as a $H_{2,n,l}$. 

If $\Gamma_1 = G_{2,n,1}$ and $\Gamma_2 = G_{2,n',1}$ have the same
zeta function, then the have the same number of edges, so $n+1 = n'+1$,
whence $n' = n$ and $\Gamma_1$ and $\Gamma_2$ are isomorphic.
If $\Gamma_1 = H_{2,n,l}$ and $\Gamma_2 = H_{2,n',l'}$ have the same
zeta function, then they have the same number of edges, so $2+n+l = 2+n'+l'$.
They also have the same number of trees (see Section~\ref{sec:kappa}) so 
that $2n = 2n'$. This implies that $\Gamma_1$ and $\Gamma_2$ are 
isomorphic.
\end{proof}

\section{Matrix determinant formulas}
\label{sec:Mdet}

In this section, we review some basic lemmas for calculating matrix determinants. We assume these are standard and leave most proofs to the reader.

\begin{lemma}
\label{lem:MkA}
Let $A$ be an $n\times n$ matrix, and $k$ be a scalar. Then
$$det(kA) = k^n\det(A).$$
\end{lemma}

\begin{definition} Let $J_{m,n}$ denote the $m\times n$ matrix of all ones, $J_n$ denote the $n\times n$ matrix of all ones, and $\vec{1} = J_{n,1}$.
\end{definition}

\begin{lemma}
If $n\geq 2$, then $\det(J_n) = 0$. Otherwise, if $n=1$ then $\det(J_n)=1$.
\end{lemma}

\begin{lemma}
\label{lem:MSylv}
For $n\times n$ matrices, the diagonal entries all equal the constant $a+b$ and all other entries are $b$, we have the following determinant formula:
$$ 
\det{\left(\begin{array}{cccccc}
a+b & b & b & b & \dots & b\\
b & a+b & b & b &  \dots & b\\
b & b & a+b & b &  \dots & b\\
b & b & b & a+b &  \dots & b\\
 &  \vdots &   &  & \ddots \\
b & b & b & b & \dots & a+b \\
\end{array}\right)}
= a^{n-1}(a+bn)
$$
In other words, $\det(aI_n+bJ_n) = a^{n-1}(a+bn)$.
\end{lemma}

\begin{proof}
If $n=1$, then $\det(aI_n+bJ_n) = a+b$, which agrees with $a^{1-1}(a+b(1)) = a+b$. Assume $n\geq 2$ henceforth. If $a=0$, then $\det(aI_n+bJ_n) = \det(bJ_n) = b^n\det(J_n)=0$ by the previous lemma because $n\geq 2$. This agrees with our formula since $(0)^{n-1}(0+n)=0$. Otherwise, if $a\neq 0$, this matrix can be written as:
$$aI_n + b\vec{1} \vec{1}^T$$
Recall Sylvester's Determinant Identity:
$$\det(I_n + u v^T) = 1 + v^T u$$
Therefore, using the Lemma~\ref{lem:MkA}:
\begin{align*}
\det\left(aI_n + b\vec{1} \vec{1}^{T}\right) &= a^n \det\left( I_n + \dfrac{b\vec{1} \vec{1}^{T}}{a}\right) \\
&= a^n \left(1 + \dfrac{b\vec{1}^{T} \vec{1}}{a}\right) \\
&= a^{n-1}(a+bn)
\end{align*}
\end{proof}

\clearpage
\begin{lemma}
\label{lem:Mblock}
For the block matrix
\[
\Gamma = \begin{bmatrix} A & B \\ C & D \end{bmatrix}
\]
with $A$ an invertible $m\times m$ matrix, $B$ an $m\times n$ matrix, $C$ an $n\times m$ matrix, and $D$ an $n\times n$ matrix,
$$\det(\Gamma) = \det(A)\det(D-CA^{-1}B).$$
\end{lemma}

\begin{lemma}
\label{lem:Mblock2}
If $A=D$, $B=C$, and $A$ and $B$ are square matrices,
$$\det(\Gamma) = \det(A-B)\det(A+B)$$
\end{lemma}



\section{Ihara zeta function of five graph families}
\label{sec:var}

In this section we determine the zeta function for five families of graphs: the complete graphs, complete bipartite graphs, 
the cocktail party graphs, the $B_n$ graphs, and the M\"obius ladders. 
For $n$ even, $B_n$ is the bipartite graph obtained by deleting a perfect matching from $K_{n/2,n/2}$.

We remark that it is well known (for example, see \cite[Example 2.3]{T}) that the Ihara zeta function of 
$C_n$, the cycle graph on $n$ vertices, is 
$$\zeta_{C_n}(u)^{-1}= (1-u^n)^2.$$

The formulas in this section have simple generalizations to the case where we add an equal number of loops
at each vertex. 
For example, we state this generalization explicitly after our first theorem about complete graphs.

\begin{theorem} Let $n \geq 3$.
    The Ihara zeta function of the complete graph $K_n$ is:
    $$\zeta_{K_n}(u)^{-1}=\left(1-u^2\right)^{n(n-3)/2}\left(1+u + (n-2)u^2\right)^{n-1}\left(1+(1-n)u + (n-2)u^2\right)$$
\end{theorem}

\begin{remark} If we instead consider $K_n$ with the addition of $k$ loops on every vertex, the formula becomes
$$\left(1-u^2\right)^{n(n-3)/2+nk}\left(1+(1-2k)u - (2k+n-2)u^2\right)^{n-1}\left(1-(n+2k-1)u + (2k+n-2)u^2\right),$$ 
where we've used the convention that $\A_{i,i} = 2k$ (see \cite[Definition 2.1]{T}) for the diagonal entries of the adjacency matrix.
All of the formulas in this section have similar generalizations, but hereafter we will not mention them explicitly.
\end{remark}

\begin{proof}
First, a complete graph has $n$ vertices and $\frac{n(n-1)}{2}$ edges, so we calculate 
\begin{align*}
    r-1 &= |E|-|V| \\
    &= \frac{n(n-1)}{2}-n \\
    &= \frac{n^2-n-2n}{2}\\ 
    &= \frac{n(n-3)}{2}.
\end{align*}
Next, the adjacency matrix will be all ones save for the diagonal (as there are no loops or multiedges). Thus $\A=J_n-I_n$. Next, each vertex has degree $n-1$, so $\QQ=(n-2)I_n$. Now, we use the determinant formula for the Ihara zeta function:
\begin{align*}
I - \A u + \QQ u^2 &= I_n - u(J_n-I_n) + (n-2)u^2I_n \\
&= (1+u+ (n-2)u^2)I_n - uJ_n
\end{align*}
We then use Lemma~\ref{lem:MSylv} to simplify:
\begin{align*}
\det(I - \A u + \QQ u^2) &= \det((1+u+ (n-2)u^2)I_n - uJ_n) \\
&= \left(1+u + (n-2)u^2\right)^{n-1}\left(1+u + (n-2)u^2-un\right) \\
&= \left(1+u + (n-2)u^2\right)^{n-1}\left(1+(1-n)u + (n-2)u^2\right)
\end{align*}

\begin{align*}
\zeta_{K_n}(u)^{-1}&= \left(1-u^2\right)^{r-1} \det\left(I - \A u + \QQ u^2\right) \\
&= \left(1-u^2\right)^{n(n-3)/2}\left(1+u + (n-2)u^2\right)^{n-1}\left(1+(1-n)u + (n-2)u^2\right)
\end{align*}
\end{proof}

\begin{theorem} Let $m,n \geq 2$.
    The Ihara zeta function of the complete bipartite graph $K_{m,n}$ is:
\begin{align*}
\zeta_{K_{m,n}}(u)^{-1}&=(1-u^2)^{mn-m-n} \\
& \hspace{0.5 in} \cdot\left[((m-1)u^2+1)^n((n-1)u^2+1)^m \right.\\
& \hspace{1.2 in} 
\left. -mnu^{2}((m-1)u^2+1)^{n-1}((n-1)u^2+1)^{m-1}\right]
\end{align*}
\end{theorem}

We see that this polynomial is symmetric under the exchange of $m$ and $n$, that the largest degree is $2mn = 2|E|$ 
(two times the number of edges), and that there are no odd degree terms in accord with Theorem~\ref{thm:even}.

\begin{proof}
Let $K_{m,n}$ be a complete bipartite graph with vertices partitioned into two sets with respective sizes $m$ and $n$.
Such a graph has $|V|=m+n$ and $|E|=mn$, so $r-1 = mn-m-n$. Using the determinant formula, this means its adjacency matrix (which is size $(m+n) \times (m+n)$) can be arranged such that it has four quadrants, the top left of size $m \times m$ and bottom right of size $n \times n$ that both consist solely of zeros, with the remaining two having $mn$ entries that are all ones. When arranged with vertices in the same order, $\QQ$ will have $m$ diagonal entries with values $n-1$ and the remaining $n$ diagonal entries with values $m-1$. We would need to find the determinant of the following matrix, and then multiply it by $(1-u^2)^{r-1}$:

{\tiny\[
\begin{vmatrix} 
    (n-1)u^2+1 & 0 & \dots & 0 & -u & -u & \dots & -u \\
    0 & (n-1)u^2+1 & & \vdots & \vdots & & \ddots & \vdots \\
    \vdots & & \ddots & 0  & -u & -u & \dots & -u \\
    0 & \dots & 0 & (n-1)u^2+1  & -u & -u & \dots & -u \\
    -u & \dots & -u & -u & (m-1)u^2+1 & 0 & \dots & 0 \\
    -u & \dots & -u & -u & 0 & (m-1)u^2+1 & & \vdots \\
    \vdots & \ddots & \vdots & \vdots & \vdots & & \ddots & 0 \\
    -u & \dots & -u & -u & 0 & \dots & 0 & (m-1)u^2+1
\end{vmatrix}
\]}

As in Lemma~\ref{lem:Mblock},
this is a block matrix with $A = ((n-1)u^2+1)I_m$, so $A^{-1} =((n-1)u^2+1)^{-1}I_m$, $B = -uJ_{m,n}$, $C = -uJ_{n,m}$, and $D = ((m-1)u^2+1)I_n$. First we compute
\begin{align*}
    \det(A) &= \det(((n-1)u^2+1)I_m) \\
    &= ((n-1)u^2+1)^m\det(I_m) \\
    &= ((n-1)u^2+1)^m.
\end{align*}
Then we find $D-CA^{-1}B$ and its determinant:
\begin{align*}
    D-CA^{-1}B &= D-(-uJ_{n,m})((n-1)u^2+1)^{-1}I_m)(-uJ_{m,n}) \\
    &= D-\frac{u^2}{(n-1)u^2+1}J_{n,m}I_mJ_{m,n} \\
    &= D-\frac{u^2}{(n-1)u^2+1}J_{n,m}J_{m,n} \\
    &= ((m-1)u^2+1)I_n+\frac{-mu^2}{(n-1)u^2+1}J_n
\end{align*}
We use Lemma~\ref{lem:MSylv} to simplify:
\begin{align*}
    \det(D-CA^{-1}B) &= ((m-1)u^2+1)^{n-1}\left(((m-1)u^2+1)+\frac{-mnu^{2}}{(n-1)u^2+1}\right).
\end{align*}
Finally, we substitute into the overall determinant as in Lemma~\ref{lem:Mblock},
\begin{align*}
    \zeta_{K_{m,n}}(u)^{-1} &=(1-u^2)^{mn-m-n}\det(A)\det(D-CA^{-1}B) \\
    &= (1-u^2)^{mn-m-n}((n-1)u^2+1)^m \\
    &\hspace{1 in} \left[((m-1)u^2+1)^{n-1}\left(((m-1)u^2+1)+\frac{-mnu^{2}}{(n-1)u^2+1}\right)\right] \\
    &=(1-u^2)^{mn-m-n}((n-1)u^2+1)^{m-1}((m-1)u^2+1)^{n-1} \\
    &\hspace{1.5 in} \left[((m-1)u^2+1)((n-1)u^2+1)-mnu^{2}\right] \\
    &=(1-u^2)^{mn-m-n} \\
    &\hspace{0.5 in} \left[((m-1)u^2+1)^n((n-1)u^2+1)^m \right.\\
    &\hspace{1.5 in} \left. -mnu^{2}((m-1)u^2+1)^{n-1}((n-1)u^2+1)^{m-1}\right] .
\end{align*} 
\end{proof}

\begin{definition}
For $n \geq 4$ even, the {cocktail party graph} $O_n$ is the complement of a perfect matching on $n$ vertices.
\end{definition}

The symbol refers to the fact that $O_6$ is the Octahedral graph, see~\cite{M}.
Note that $O_n = K_{2,2, \ldots, 2}$, the 
complete multipartite graph with $n/2$ parts, each consisting of two vertices.

\begin{theorem} Let $n \geq 2$.
    The Ihara zeta function for $O_{2n}$ is:
    \begin{align*}
\zeta_{O_{2n}}(u)^{-1}&= (1-u^2)^{2n^2-4n}\left((2n-3)^2u^4+(4n-6)u^3+(4n-6)u^2+2u+1\right)^{n-1} \\
&\ \ \ \cdot((2n-3)u^2+1)((2n-3)u^2+(2-2n)u+1)
\end{align*}

\end{theorem}
\begin{proof}
Since $O_{2n}$ is the complement of a perfect matching, 
the number of edges in $O_{2n}$ will be the amount in $K_{2n}$ minus $n$.
\begin{align*}
r-1 &= |E|-|V|\\
&= \frac{2n(2n-1)}{2}-n-2n\\
&= 2n^2-n-3n \\
&=2n^2-4n
\end{align*}
Note that $|E|=2n^2-2n=2n(n-1)$. The adjacency matrix will be mostly ones, save for the diagonal and $2n$ other entries. We can arrange the vertices such that we use the first vertex from all of the sets, then use the second vertex of all of the sets with the sets in the same order. This will create the following block matrix:
$$\A=
\begin{bmatrix}
J_n-I_n & J_n-I_n \\
J_n-I_n & J_n-I_n
\end{bmatrix}
$$
As vertices are not adjacent to themselves (diagonals of the top-left and bottom-right matrices are zero), every $i,i+n$ and $i+n,i$ entry for $i\in \{1,2...,n\}$ will be zero by construction (the diagonals of the other two matrices are zero), and all other entries are one. Each vertex is connected to all vertices except itself and one other, so their degrees will be $2n-2$. Thus $\QQ=(2n-3)I_{2n}$.

\begin{align*}
I - \A u + \QQ u^2 &= I_{2n}-u\begin{bmatrix}
J_n-I_n & J_n-I_n \\
J_n-I_n & J_n-I_n
\end{bmatrix}+(2n-3)u^2I_{2n} \\
&= \begin{bmatrix}
uI_n-uJ_n & uI_n-uJ_n \\
uI_n-uJ_n & uI_n-uJ_n
\end{bmatrix}+((2n-3)u^2+1)I_{2n} \\
&= \begin{bmatrix}
uI_n-uJ_n & uI_n-uJ_n \\
uI_n-uJ_n & uI_n-uJ_n
\end{bmatrix}+\begin{bmatrix}
((2n-3)u^2+1)I_n & 0 \\
0 & ((2n-3)u^2+1)I_n
\end{bmatrix} \\
&= \begin{bmatrix}
uI_n-uJ_n+((2n-3)u^2+1)I_n & uI_n-uJ_n \\
uI_n-uJ_n & uI_n-uJ_n+((2n-3)u^2+1)I_n
\end{bmatrix} \\
&= \begin{bmatrix}
((2n-3)u^2+u+1)I_n-uJ_n & uI_n-uJ_n \\
uI_n-uJ_n & ((2n-3)u^2+u+1)I_n-uJ_n
\end{bmatrix} \\
\end{align*}
Let $a=((2n-3)u^2+u+1)$. We see this is a block matrix with $A=D=aI_n-uJ_n$, and $B=C=uI_n-uJ_n$. Thus, by Lemma~\ref{lem:Mblock2}, the determinant will equal $\det(A-B)\det(A+B)$. Since $A-B = (a-u)I_n$, by Lemma~\ref{lem:MkA} its determinant is $(a-u)^n$, and by Lemma~\ref{lem:MSylv} the determinant of $A+B=-2uJ_n+(a+u)I_n$ is $(a+u)^{n-1}(a+u-2nu)$.

\begin{align*}
\det(I - \A u + \QQ u^2) &=\det(A-B)\det(A+B) \\
&= (a-u)^n(a+u)^{n-1}(a+(1-2n)u) \\
&= (a^2-u^2)^{n-1}(a-u)(a+(1-2n)u) \\
\end{align*}

As a quick aside,
\begin{align*}
a^2 &= ((2n-3)u^2+u+1)^2 \\ 
&= (2n-3)^2u^4+u^2+1+2(2n-3)u^3+2(2n-3)u^2+2u\\
&= (2n-3)^2u^4+(4n-6)u^3+(4n-5)u^2+2u+1.
\end{align*}

Returning to the main equation,
\begin{align*}
\det(I - \A u + \QQ u^2) &= \left((2n-3)^2u^4+(4n-6)u^3+(4n-5)u^2+2u+1-u^2\right)^{n-1} \\
&\ \ \ \cdot((2n-3)u^2+u+1-u)((2n-3)u^2+u+1+(1-2n)u)\\
&= \left((2n-3)^2u^4+(4n-6)u^3+(4n-6)u^2+2u+1\right)^{n-1} \\
&\ \ \ \cdot((2n-3)u^2+1)((2n-3)u^2+(2-2n)u+1)\\
\end{align*}

\begin{align*}
\zeta_{O_{2n}}(u)^{-1} &= (1-u^2)^{2n^2-4n}\det(I-\A u + \QQ u^2) \\
&= (1-u^2)^{2n^2-4n}\left((2n-3)^2u^4+(4n-6)u^3+(4n-6)u^2+2u+1\right)^{n-1} \\
&\ \ \ \cdot((2n-3)u^2+1)((2n-3)u^2+(2-2n)u+1) 
\end{align*}
\end{proof}

Recall that $B_{2n}$ is $K_{n,n}$ after deleting a perfect matching.

\begin{theorem} Let $n\geq 3$.
The Ihara zeta function for $B_{2n}$ is:
$$
\zeta_{B_{2n}}(u)^{-1} 
= (1-u^2)^{n(n-3)}((1+(n-2)u^2)^2-u^2)^{n-1}((1+(n-2)u^2)^2-(1-n)^2u^2)
$$
\end{theorem}

\begin{remark} Note that, in accord with Theorem~\ref{thm:even},
the zeta polynomial for the bipartite graph $B_{2n}$ is even.
\end{remark}

\begin{proof}
Let $v_1, v_2, \ldots, v_{n}$ and $w_1, w_2, \ldots, w_{n}$ be the two parts of $V$ and assume that 
$v_iw_i \not\in E$. First, $r -1= |E| - |V| = (n-1)n - 2n = (n-3)n$.
Since $B_{2n}$ is regular of degree $n-1$,
$I-\A u +\QQ u^2$ is a block matrix (see Lemma~\ref{lem:Mblock}) with $A = D = (1+(n-2)u^2) I_n$ and $B = C= -uJ_n + uI_n$. Since we have this symmetry, we can use the special case Lemma~\ref{lem:Mblock2} along with Lemma~\ref{lem:MSylv} to simplify, handling both parts simultaneously:
\begin{align*}
\det(A\pm B) &= \det((1+(n-2)u^2) I_n \pm(-uJ_n + uI_n)) \\
&= \det((1\pm u+(n-2)u^2)I_n \mp uJ_n) \\
&= (1\pm u+(n-2)u^2)^{n-1}((1\pm u+(n-2)u^2)\mp un) \\
&= (1\pm u+(n-2)u^2)^{n-1}((1\pm(1-n)u+(n-2)u^2))
\end{align*}

\begin{align*}
\det(I-\A u +\QQ u^2) &= \det(A-B)\det(A+B) \\
&= (1+ u+(n-2)u^2)^{n-1}((1+(1-n)u+(n-2)u^2)) \\
&\ \ \ \cdot(1- u+(n-2)u^2)^{n-1}((1-(1-n)u+(n-2)u^2)) \\
&= ((1+(n-2)u^2)^2-u^2)^{n-1}((1+(n-2)u^2)^2-(1-n)^2u^2) \\
\end{align*}
\begin{align*}
\zeta_{B_{2n}}(u)^{-1} &= (1-u^2)^{n(n-3)} \det(I-\A u +\QQ u^2)\\
 &= (1-u^2)^{n(n-3)}((1+(n-2)u^2)^2-u^2)^{n-1}((1+(n-2)u^2)^2-(1-n)^2u^2) \\
\end{align*}

\end{proof}

For $n$ even, the M\"obius ladder~\cite{GH} graph $M_n$ consists of an $n$-cycle together with $n/2$ diagonal edges that connect 
antipodal vertices on the cycle. It is regular of degree 3.

\begin{theorem} 
Let $n\geq 4$ be even and $\omega = e^{2 \pi i/n}$.
The Ihara zeta function for the M\"{o}bius ladder graph $M_n$ is:
$$
\zeta_{M_{n}}(u)^{-1} = 
(1-u^2)^{n/2} u^n  \prod_{k=0}^{n-1} \left(-\frac{1+2u^2}{u} +  \omega^k + \omega^{kn/2} + \omega^{k(n-1)}\right)
$$
\end{theorem}

\begin{remark} 
The proof below can be generalized to other circulant graphs.
\end{remark}

\begin{proof}
If we label the vertices of the cycle $v_1, v_2, \ldots, v_n$ in order, the adjacency matrix has a nice 
structure. For example, when $n = 10$, we have

$$\A = \left(\begin{array}{llllllllll}
0 & 1 & 0 & 0  & 0 & 1 & 0 & 0 & 0 & 1 \\
1 & 0 & 1 & 0 & 0  & 0 & 1 & 0 & 0 & 0 \\
0 & 1 & 0 & 1 & 0 & 0  & 0 & 1 & 0 & 0 \\
0 & 0 & 1 & 0 & 1 & 0 & 0  & 0 & 1 & 0 \\
0 & 0 & 0 & 1 & 0 & 1 & 0 & 0  & 0 & 1 \\
1 & 0 & 0 & 0 & 1 & 0 & 1 & 0 & 0  & 0 \\
0 & 1 & 0 & 0 & 0 & 1 & 0 & 1 & 0 & 0  \\
0 & 0 & 1 & 0 & 0 & 0 & 1 & 0 & 1 & 0  \\
0 & 0 & 0 & 1 & 0 & 0  & 0 & 1 & 0  & 1 \\
1 & 0 & 0 & 0 & 1 & 0 & 0  & 0 & 1 & 0  
\end{array}\right)$$

\smallskip

There are diagonals of 1's above and below the main diagonal as well as diagonals of 1's in the top right block and lower left block 
of the matrix. The matrix $-\frac{1}{u} (I - \A u + \QQ u^2)$ is similar with $a = -(1+2u^2)/u$ replacing the 0's on the diagonal. This is a circulant
matrix and, as explained by Lehmer~\cite{L}, its determinant can be written
$$ \det \left(-\frac{1}{u} (I - \A u + \QQ u^2) \right)  = \prod_{k=0}^{n-1} (a +  \omega^k + \omega^{kn/2} + \omega^{k(n-1)}).$$

Applying Lemma~\ref{lem:MkA} to Definition~\ref{def:3D}, we obtain the required equation.
\end{proof}

\section{Graphs of small order}
\label{sec:ord}

In this section we determine the Ihara zeta function for graphs of small order.
Let $BQ_a$ denote the multigraph on one vertex with $a$ loops, sometimes called a
bouquet of $a$ loops. Then $\A = 2a$,
$\QQ = 2a-1$, $r-1 = |E| - |V| = a-1$, and 
$$\zeta_{BQ_{a}}(u)^{-1} = (1-u^2)^{a-1} (1 -2au + (2a-1)u^2)$$ 
In particular, if $a = 0$, $\zeta_{BQ_{0}}(u)^{-1} = \zeta_{K_{1}}(u)^{-1} = 1$.

Let $D_{a,b,c}$ denote a two vertex multigraph, with $a$ loops on $v_1$, $b$ loops 
on $v_2$ and $c$ $v_1v_2$ edges. Then, 
$$I - \A u + \QQ u^2 = \left(\begin{array}{cc}
1 - 2au + (2a+c-1)u^2 & -cu \\
-cu & 1 - 2bu + (2b+c-1)u^2
\end{array}\right)$$
has determinant $(1-2au+(2a+c-1)u^2)(1-2bu+(2b+c-1)u^2) - (cu)^2$.
Since $r-1 = |E|-|V| = a+b+c -2$,
$$\zeta_{D_{a,b,c}}(u)^{-1} = (1-u^2)^{a+b+c-2} \left((1-2au+(2a+c-1)u^2)(1-2bu+(2b+c-1)u^2) - (cu)^2 \right).$$
Note that the graph $BL$ mentioned in Section~\ref{sec:rk2}, consisting of a bigon and a loop, is $D_{0,1,2}$.

Let $T_{\vec{a}}$, where $\vec{a} = (a_1,a_2,a_3,b_{12},b_{13},b_{23})$, denote the multigraph on three vertices with
$a_i$ loops on $v_i$ ($i = 1,2,3$), and $b_{ij}$ $v_iv_j$ edges. Here,   $I - \A u + \QQ u^2$ is
{\tiny$$ \left(\begin{array}{ccc}
1 - 2a_1u + (2a_1+b_{12}+b_{13}-1)u^2 & -b_{12}u & -b_{13}u \\
-b_{12}u & 1 - 2a_2u + (2a_3+b_{12}+b_{23}-1)u^2 & -b_{23}u \\
-b_{13}u & -b_{23}u & 1 - 2a_3u + (2a_3+b_{13}+b_{23}-1)u^2
\end{array}\right)$$}
and $r-1 = |E| - |V| = a_1+a_2+a_3+b_{12}+b_{13}+b_{23}-3$.
Thus, one can calculate $\zeta_{T_{\vec{a}}}(u)^{-1} = (1-u^2)^{r-1} \det(I - \A u + \QQ u^2)$.
Note that the graph $BB$ mentioned in Section~\ref{sec:rk2}, consisting of two bigons that share a vertex, is the 
graph $T_{\vec{a}}$,  where $\vec{a} = (0,0,0,0,2,2)$. In this case,
$$I - \A u + \QQ u^2 = 
\left(\begin{array}{ccc}
1 + u^2 & 0 & -2u \\
 0 & 1 + u^2 & -2u \\
-2u & -2u & 1 + 3u^2
\end{array}\right)$$
with determinant $3u^6 - u^4 -3u^2 + 1$. The rank is $r = 2$ so that 
$\zeta_{BB}(u)^{-1} = (1-u^2)(3u^6 - u^4 -3u^2 + 1)$.

Next, we turn to simple graphs on four vertices. Given the requirements that $G$ be connected with no vertex
of degree 1, there are only three graphs. We gave the zeta function of $K_4$ and $C_4$ in the previous section.
For $K_4^-$, that is, $K_4$ with one edge removed, we have
$$\zeta_{K_4^-} (u)^{-1} = 
-4u^{10} + u^8 + 4u^7 + 4u^6 - 2u^4 - 4u^3 + 1.$$

For simple graphs on five vertices, we determined zeta functions for $K_5$, $K_{2,3}$, and $C_5$ in the previous section.
There is only one graph of order five and size nine, $K_5^-$, for which
\begin{align*}
\zeta_{K_5^-} (u)^{-1}  &= 
108u^{18} - 360u^{16} - 80u^{15} + 345u^{14} + 252u^{13} + 52u^{12} - 222u^{11} - 234u^{10}\\
 & \ \ \ \ - 32u^9 + 69u^8  + 108u^7 + 37u^6 - 12u^5 - 18u^4 - 14u^3 + 1.
 \end{align*}

There are two graphs of size eight, which we can describe in terms of the complement. If the complement
of $G$ is a matching on four vertices, then
$$\zeta_{G} (u)^{-1} = 
-48u^{16} + 112u^{14} + 32u^{13} - 40u^{12} - 64u^{11} - 68u^{10} + 8u^9 + 41u^8 + 40u^7 + 12u^6 - 8u^5 - 10u^4 - 8u^3 + 1.
$$
On the other hand, if the complement of $G$ is a path on three vertices, 
$$\zeta_{G} (u)^{-1} = 
-36u^{16} + 73u^{14} + 28u^{13} - 4u^{12} - 50u^{11} - 62u^{10} - 8u^9 + 17u^8 + 44u^7 + 21u^6 - 4u^5 - 10u^4 - 10u^3 + 1.
$$

There are three graphs of size seven. If the complement is a triangle,
$$\zeta_{G} (u)^{-1} = 
9u^{14} - 4u^{12} - 6u^{11} - 18u^{10} + 9u^8 + 12u^7 + 9u^6 - 6u^4 - 6u^3 + 1.
$$
When the complement is a path on four vertices,
$$\zeta_{G} (u)^{-1} = 
12u^{14} - 11u^{12} - 10u^{11} - 11u^{10} + 6u^9 + 6u^8 + 12u^7 + 7u^6 - 2u^5 - 4u^4 - 6u^3 + 1.$$
Finally, when the complement is a path on three vertices with a disjoint $K_2$,
$$\zeta_{G} (u)^{-1} = 
16u^{14} - 20u^{12} - 8u^{11} - 12u^{10} + 4u^9 + 17u^8 + 12u^7 + 4u^6 - 4u^5 - 6u^4 - 4u^3 + 1.$$

There is only one graph of size at most five, $C_5$. It remains to
enumerate those of size six. One is $K_{2,3}$ and there are two others.
If the complement of $G$ is a path on four vertices, then
$$\zeta_{G} (u)^{-1} = 
-4u^{12} + u^{10} + 2u^9 + 3u^8 + 2u^7 + u^6 - 2u^5 - 2u^4 - 2u^3 + 1.$$
Finally, if the complement of $G$ is a $4$-cycle, then
$$\zeta_{G} (u)^{-1} = 
-3u^{12} + 4u^9 + 2u^6 - 4u^3 + 1.$$

\section{Spanning trees}
\label{sec:kappa}

Let $\kappa_G$ denote the number of spanning trees of graph $G$.
In this section, we use the Ihara zeta function to verify $\kappa_G$ for
various graphs discussed in this paper. 

The following theorem is given as an exercise in \cite{T}.
 
\begin{theorem}
\label{thm:ASTF}
The Ihara zeta function satisfies
$$ \frac{d^r}{du^r} \zeta_G(u)^{-1}\bigg\rvert_{u = 1} = (-1)^{r-1} 2^r r! (r-1) \kappa_G,$$
where $r = |E| - |V| + 1$ is the graph's rank.
\end{theorem}

We first check $\kappa_G$ for the graphs of rank two.
If $r = 2$, the Theorem states 

$$ \kappa_G = - \frac{1}{8} \frac{d^2}{du^2} \zeta_G(u)^{-1}\bigg\rvert_{u = 1}$$

It's easy to see that the $\kappa_{G_{m,n}} = mn$ since a 
spanning tree is formed by removing one edge of the $m$-cycle
and one edge of the $n$ cycle. This agrees with the result of Theorem~\ref{thm:ASTF}
using the formula for $\zeta_{G_{m,n}}(u)^{-1}$ given by Theorem~\ref{thm:Gmn}.

Similarly, $\kappa_{H_{m,n,l}} = mn$, which agrees with the result of Theorem~\ref{thm:ASTF}
using the formula for $\zeta_{H_{m,n,l}}(u)^{-1}$ of Theorem~\ref{thm:Hmnl}.

The graph $G_{m,n,p}$ is a theta graph: a union of three internally disjoint paths
of length $m-p$, $n-p$, $p$ between the distinct vertices $v$ and $w$. A spanning tree is
formed by removing one edge from two of the three paths. Thus,
\begin{align*}
    \kappa_{G_{m,n,p}} & = p(m-p) + p(n-p) + (m-p)(n-p) \\
    & = mn - p^2.
\end{align*}
This agrees with the result of applying Theorem~\ref{thm:ASTF} to the equation
for $\zeta_{G_{m,n,p}}(u)^{-1}$ in Theorem~\ref{thm:Gmnp}.

Next we verify known formulas for the number of spanning trees of four of the graph
families in Section~\ref{sec:var}. In each case, we found
$\zeta_{G}(u)^{-1} = (u-1)^{r-1}f(u)$ for some polynomial $f(u)$.
It follows that 
$$ \frac{d^r}{du^r} \zeta_G(u)^{-1}\bigg\rvert_{u = 1} = r! f'(u) \bigg\rvert_{u = 1}.$$
Comparing with Theorem~\ref{thm:ASTF}, 
\begin{equation}
\label{eqn:KG}
\kappa_G = \frac{(-1)^{r-1}}{2^r(r-1)} f'(u) \bigg\rvert_{u = 1}.
\end{equation}

For $K_n$, $f(u)$ is a product of three factors, the third of which, $\left(1+(1-n)u + (n-2)u^2\right)$, evaluates to zero at $u = 1$.
Thus,
\begin{align*}
f'(u) \bigg\rvert_{u = 1} &= \left(-1-u\right)^{n(n-3)/2}\left(1+u + (n-2)u^2\right)^{n-1}\left(1+(1-n)u + (n-2)u^2\right)' \bigg\rvert_{u = 1} \\
 &= (-2)^{n(n-3)/2}n^{n-1} \left( (1-n) + 2(n-2) \right) \\
 &= (-2)^{n(n-3)/2}n^{n-1} (n-3) 
\end{align*}

Substituting into Equation~\ref{eqn:KG}, and using $r-1 = n(n-3)/2$, we have
\begin{align*}
\kappa_{K_n} &= \frac{(-1)^{r-1}}{2^r(r-1)} f'(u) \bigg\rvert_{u = 1}\\
 &= \frac{2^{r-1}n^{n-1} (n-3)}{2^r n (n-3)/2} \\
 & = n^{n-2},
\end{align*}
which is Cayley's formula~\cite{C}.

For $K_{m,n}$, 
\begin{align*}
f(u) &=
(-1-u)^{r-1} \left[((m-1)u^2+1)^n((n-1)u^2+1)^m \right.\\
& \hspace{1.2 in} 
\left. -mnu^{2}((m-1)u^2+1)^{n-1}((n-1)u^2+1)^{m-1}\right]
\end{align*}

Since the second factor evaluates to zero at $u = 1$, 
\begin{align*}
f'(u) \bigg\rvert_{u = 1} &= (-2)^{r-1}\left[nm^{n-1}2(m-1)n^m + m^n mn^{m-1}2(n-1)\right.\\
& \hspace{0.2 in} 
\left. -mn \left(2 m^{n-1}n^{m-1} + (n-1)m^{n-2} 2(m-1) n^{m-1} + m^{n-1} (m-1)n^{m-2} 2(n-1)\right)\right]\\
&= (-2)^{r-1}\left[2(m-1) \left( m^{n-1}n^{m+1} - (n-1) m^{n-1}n^{m} \right) \right.\\
& \hspace{1.2 in} \left.  + 2(n-1) \left( m^{n+1}n^{m-1} - (m-1)m^nn^{m-1} \right) - 2m^nn^m \right]\\
&= (-2)^{r-1}\left[2(m-1) m^{n-1}n^m + 2(n-1) m^nn^{m-1} - 2m^nn^m \right] \\
&= (-2)^{r-1}\left[2m^nn^m - 2m^{n-1}n^{m-1}(m+n)\right] \\
&= (-2)^{r-1}m^{n-1}n^{m-1} 2(mn - m-n)\\
\end{align*}

Substituting into Equation~\ref{eqn:KG}, and using $r-1 = mn-m-n$, we have
\begin{align*}
\kappa_{K_{m,n}} &= \frac{(-1)^{r-1}}{2^r(r-1)} f'(u) \bigg\rvert_{u = 1} = m^{n-1}n^{m-1},
\end{align*}
also a well-known formula (see for example~\cite{A}).

The Ihara zeta function for $O_{2n}$ is a product of four factors, the last of which
evaluates to zero at $u = 1$. Thus,
\begin{align*}
f'(u) \bigg\rvert_{u = 1} &= (-2)^{r-1} \left((2n-3)^2+(4n-6)+(4n-6)+3\right)^{n-1} \\
 &\ \ \ \cdot (2n-2) \left(2(2n-3)+(-2n+2)\right)  \\
 &= (-2)^{r-1} (4n^2-4n)^{n-1}  (2n-2)(2n-4)
\end{align*}

Substituting into Equation~\ref{eqn:KG}, and using $r-1 = 2n^2-4n$, we have
\begin{align*}
\kappa_{O_{2n}} &= \frac{(-1)^{r-1}}{2^r(r-1)} f'(u) \bigg\rvert_{u = 1} \\
 &= \frac{(4n^2-4n)^{n-1}  (2n-2)(2n-4)}{2(2n^2-4n)}\\
 &= \frac{\left(4n(n-1)\right)^{n-1}  4(n-2)(n-1)}{4n(n-2)}\\
 &= 4^{n-1}n^{n-2}(n-1)^n
 \end{align*}
 in agreement with \cite[Corollary 3.5]{M}.
 
 Finally, for $B_{2n}$ we have a similar argument where the last factor of $f(u)$, 
 $((1+(n-2)u^2)^2-(1-n)^2u^2)$,
 evaluates
 to zero at $u = 1$. This means
 \begin{align*}
f'(u) \bigg\rvert_{u = 1} &= (-2)^{r-1} ((n-1)^2-1)^{n-1} \left(4(n-1)(n-2) - 2(1-n)^2 \right)\\ 
 &= (-2)^{r-1} (n^2-2n)^{n-1}2(n-1)(n-3)
\end{align*}  

and, since $r-1 = n(n-3)$,
\begin{align*}
\kappa_{B_{2n}} &= \frac{(-1)^{r-1}}{2^r(r-1)} f'(u) \bigg\rvert_{u = 1} \\
 &= (n^2-2n)^{n-1}(n-1)/n \\
 &= n^{n-2}(n-2)^{n-1}(n-1),
\end{align*}
which agrees with \cite[Corollary 3.10]{M}.

\end{document}